\documentclass{amsart}

\usepackage{amssymb,amsmath}

\theoremstyle{plain}
\newtheorem{theorem}{Theorem}[section]
\newtheorem*{theorem*}{Theorem}
\newtheorem{lemma}[theorem]{Lemma}

\newtheorem{proposition}[theorem]{Proposition}
\theoremstyle{remark}
\newtheorem{remark}[theorem]{Remark}
\numberwithin{equation}{section}
\theoremstyle{definition}
\newtheorem{definition}[theorem]{Definition}
\numberwithin{equation}{section}
\numberwithin{equation}{section}

\newcommand\RR{\mathbb{R}}
\newcommand\dir{{\rm{\,d}}\rho}
\newcommand\dil{{\rm{\,d}}\lambda}
\newcommand\di{{\rm{\,d}}}
\newcommand\nep{{\rm{e}}}
\newcommand\CZ{Calder\'on--Zygmund }

\begin{document}

\title[Heat maximal function]
{ Heat maximal function \\on a Lie group of exponential growth }



\thanks{Work partially supported by Progetto GNAMPA 2011``Analisi armonica su gruppi di Lie e variet\'a'' .}
% the \thanks layout looks crappy!

\author[P. Sj\"ogren and M. Vallarino]
{Peter Sj\"ogren and Maria Vallarino}

%\address{Dipartimento di Matematica \\
%Universit\`a di Genova \\ via Dodecaneso 35, 16146 Genova \\ Italia}

\address{Peter Sj\"ogren: 
Mathematical Sciences, 
University of Gothenburg and  Mathematical Sciences, 
Chalmers\\
S-412 96 G\"oteborg\\Sweden  
}
\email{peters@math.chalmers.se}

\address{Maria Vallarino:
Dipartimento di Matematica 
\\ Politecnico di Torino\\
Corso Duca degli Abruzzi, 24 \\ 10129 Torino \\ Italy} \email{maria.vallarino@polito.it}

\begin{abstract}
Let $G$ be the Lie group ${\Bbb{R}}^2\rtimes {\Bbb{R}}^+$ endowed with the Riemannian symmetric space structure. Let $X_0,\, X_1,\,X_2$ be a distinguished basis of left-invariant vector fields of the Lie algebra of $G$ and define the Laplacian $\Delta=-(X_0^2+X_1^2+X_2^2)$. In this paper, we show that the maximal function associated with the heat kernel of the Laplacian $\Delta$  is bounded from the Hardy space $H^1$ to $L^1$. We also prove that the heat maximal function does not provide a maximal characterization of the Hardy space $H^1$. 
\end{abstract}

\maketitle

\section{Introduction}\label{intro}
Let $G$ be the Lie group $\RR^2\rtimes \RR^+$ where the product rule is the following:
$$(x_1,x_2,a)\cdot(x'_1,x'_2,a')=(x_1+a\,x'_1,x_2+a\,x'_2,a\,a') $$
for $(x_1,x_2,a),\,(x'_1,x'_2,a')\in G$. We shall denote by $x$ the point $(x_1,x_2,a)$. The group $G$ is not unimodular; the right and left Haar measures are given by
$$\dir(x)=a^{-1}\,\di x_1\di x_2\di a\qquad{\rm{and}}\qquad \dil(x)=a^{-3}\,\di x_1\di x_2\di a\,,$$
respectively. The modular function is thus $\delta(x)=a^{-2}$. Throughout this paper, unless explicitly stated, we use the right measure $\rho$ on $G$ and denote by $L^p$, $\|\cdot\|_p$ and $\langle\cdot,\cdot \rangle$  the $L^p$-space, the $L^p$-norm and the $L^2$-scalar product with respect to the measure $\rho$. 

The group $G$ has a Riemannian symmetric space structure, and the corresponding metric, which we denote by $d$, is that of the three-dimensional hyperbolic half-space. The metric $d$ is invariant under left translation and given by
\begin{equation}\label{metrica}
\cosh r(x)=\frac{a+a^{-1}+a^{-1}|x|^2}{2}\,,
\end{equation}
where $r(x)=d(x,e)$ denotes the distance of the point $x$ from the identity $e=(0,0,1)$ of $G$ and $|x|$ denotes the Euclidean norm of the two-dimensional vector $(x_1,x_2)$, i.e., $|x|=\sqrt{x_1^2+x_2^2}$. The measure of a hyperbolic ball $B_r$, centered at the identity and of radius $r$, behaves like
$$\lambda(B_r)=\rho(B_r)\sim  \begin{cases}
r^3&{\rm{if}}~r<1\\
\nep^{2r}&{\rm{if}}~r\geq 1\,.
\end{cases}
$$
Thus $G$ is a group of {\emph{exponential growth}}. In this context, the classical \CZ theory and the classical definition of the atomic Hardy space $H^1$ (see \cite{CW1, CW2, St}) do not apply. But W.~Hebisch and T.~Steger \cite{HS} have constructed a \CZ theory which applies to some spaces of exponential growth, in particular to the space $(G,d,\rho)$ defined above. The main idea is to replace the family of balls which is used in the classical \CZ theory by a suitable family of parallelepipeds which we call \emph{\CZ sets}. The definition appears in \cite{HS} and implicitly in \cite{GS}, and reads as follows.
\begin{definition}\label{Czsets}
A Calder\'on-Zygmund set is a parallelepiped $R=[x_1-L/2,x_1+L/2]\times[x_2-L/2,x_2+L/2]\times [a\nep^{-r},a\nep^r]$, where $L>0$, $r>0$ and $(x_1,x_2,a)\in G$ are related by
$$\nep^{2}a\,r\leq L< \nep^{8 }a\,r\qquad{\rm{if~}}r<1\,,$$
$$a\,\nep^{2r}\leq L< a\,\nep^{8r}\qquad{\rm{if~}}r\geq 1\,.$$
The point $(x_1,x_2,a)$ is called the center of $R$.
\end{definition}
We let $\mathcal R$ denote the family of all \CZ sets, and observe that  $\mathcal R$ is invariant under left translation. Given  $R \in \mathcal R$, we define its dilated set as $R^*=\{x\in G:~d(x,R)<r\}$. There exists an absolute constant $C_0>0$ such that $\rho(R^*)\leq C_0\,\rho(R)$ and 
\begin{equation}\label{C0}
R\subset B\big((x_1,x_2,a), C_0r   \big)\,.
\end{equation}

In \cite{HS} it is proved that every integrable function on $G$ admits a \CZ decomposition involving the family $\mathcal R$, and that a new \CZ theory can be developed in this context. By using the \CZ sets, it is natural to introduce an atomic Hardy space $H^1$ on the group $G$, as follows (see \cite{V} for details).

We define an {\emph{atom}} as a function $A$ in $L^1$ such that
\begin{itemize}
\item [(i)] $A$ is supported in a \CZ set $R$;
\item [(ii)]$\|A\|_{\infty}\leq \rho(R)^{-1}\,;$ 
\item [(iii)]$\int A\dir =0$\,.
\end{itemize}
The atomic Hardy space is now defined in a standard way. 
\begin{definition}
The Hardy space $H^{1}$ is the space of all functions $f$ in $ L^1$ which can be written as $f=\sum_j \lambda_j\, A_j$, where $A_j$ are atoms and $\lambda _j$ are complex numbers such that $\sum _j |\lambda _j|<\infty$. We denote by $\|f\|_{H^{1}}$ the infimum of $\sum_j|\lambda_j|$ over such decompositions.
\end{definition}

The \CZ theory from \cite{HS} has turned out useful for the study of the boundedness of singular integral operators related to a distinguished Laplacian on $G$, which is defined as follows.

Let $X_0,\,X_1,\,X_2$ denote the left-invariant vector fields  
$$X_0=a\,\partial_a\qquad X_1=a\,\partial_{x_1}\qquad  X_2=a\,\partial_{x_2}\,,$$
which span the Lie algebra of $G$. The Laplacian $\Delta=-(X_0^2+X_1^2+X_2^2)$ is a left-invariant operator which is essentially selfadjoint on $L^2(\rho)$. It is well known that the heat semigroup $\big( e^{-t\Delta}\big)_{t>0}$ is given by a kernel $p_t$, in the sense that $e^{-t\Delta}f=f\ast p_t$ for suitable functions $f$.  Let $\mathcal M$ denote the corresponding maximal operator, defined by
\begin{equation} \label{hetamaxop}
\mathcal M f(x)=\sup_{t>0} |f\ast p_t(x)|\qquad \forall x\in G\,.
\end{equation}
M. Cowling, G. Gaudry, S. Giulini and G. Mauceri \cite{CGGM} proved that $\mathcal M$ is of weak type $(1,1)$ and bounded on 
$L^p$ for all $p>1$. In this paper, we prove the following result. 
\begin{theorem}\label{H1L1bdd}
\begin{itemize}
\item[(i)] If $f\in H^1$, then $\mathcal Mf\in L^1$ and
$$
\|\mathcal Mf\|_1\leq C\,\|f\|_{H^1}\,,
$$ 
where $C$ is independent of $f$.
\item[(ii)] However, the converse inequality is false, and $f\in L^1$ and $\mathcal Mf\in L^1$ does not imply $f\in H^1$. 
\end{itemize}
\end{theorem}

\begin{remark}
A similar maximal operator $\mathcal {\widetilde{M}}$ can be defined by means of the Poisson semigroup $\big(e^{-t\sqrt\Delta}\big)_{t>0}$. Theorem \ref{H1L1bdd} is valid also for this maximal operator, since the 
proof given below remains valid with only small modifications. 
 
\end{remark}

In our setting, there is thus no characterization of the atomic $H^1$ space by means of the heat or the Poisson 
maximal operator. This is in contrast with the classical theory and other settings, as we now recall.

In the Euclidean case, the atomic Hardy space $H^1(\mathbb R^n)$ has a maximal characterization (see \cite[Chapter 3]{St}). More precisely, $H^1(\mathbb R^n)$ coincides with the set of functions $f\in L^1(\mathbb R^n)$ such that $\mathcal M_{\phi}f\in L^1(\mathbb R^n)$, where
$$
\mathcal M_{\phi}f(x)=\sup_{t>0}\big| f\ast \phi_t (x)\big| \qquad \forall x\in\mathbb R^n\,,
$$
$\phi$ is any Schwartz function with $\int\phi(x)dx\neq 0$ and $\phi_t(x)=t^{-n}\,\phi(t^{-1}\,x)$. Moreover, the norms $\|f\|_{H^1(\mathbb R^n)}$ and $\|\mathcal M_{\phi}f\|_{L^1(\mathbb R^n)}$ are equivalent.

On a space $\mathbb X$ of homogeneous type R.R. Coifman and G. Weiss \cite{CW1, CW2} introduced an atomic Hardy space $H^1_{\rm{at}}(\mathbb X)$. Under some additional assumption, different maximal characterizations of the Hardy space $H^1_{\rm{at}}(\mathbb X)$ were obtained by R.A. Mac\'ias and C. Segovia \cite{MS}, W.M. Li \cite{L}, L. Grafakos, L. Liu and D. Yang \cite{GLY1, GLY2}.

On the Euclidean space $\mathbb R^n$ endowed with a nondoubling Radon measure $\mu$ of polynomial growth, X.~Tolsa \cite{T} defined an atomic Hardy space $H^1(\mu)$ and proved that it can be characterized by a grand maximal operator as in the classical setting.

In the setting of a gaussian measure and the Ornstein-Uhlenbeck operator in $\mathbb R^n$, G. Mauceri and S. Meda \cite{MM} introduced an atomic $H^1$ space. For $n=1$ Mauceri, Meda and Sj\"ogren \cite{MMS} gave a maximal characterization of this space.

\bigskip

In this paper, $C$ denotes a positive, finite constant which may vary from line to line and may depend on parameters according to the context. Given two positive quantities $f$ and $g$, by $f\lesssim g$ we mean that there exists a constant $C$ such that $f\leq C\, g$, and $f\sim g$ means that $g\lesssim f\lesssim g$.

\section{The heat maximal function}\label{heat}

We shall use the following integration formula (see for instance \cite[Lemma 1.3]{CGHM}): for any radial function $f$ such that $\delta^{1/2}f$ is integrable 
\begin{equation}\label{intformula}
\int_G\delta^{1/2}f\dir=\int_0^{\infty}f(r)\,r\,\sinh r\di r\,.
\end{equation}

The convolution in $G$ is defined by 
\begin{equation}\label{conv1}
f\ast g(x)=\int f(xy^{-1})g(y)\dir (y)\,.
\end{equation}

Let $p_t$ denote the heat kernel associated with $\Delta$. It is well known \cite[Theorem 5.3, Proposition 5.4]{CGGM}, \cite[formula (5.7)]{ADY} that 
\begin{equation}\label{heatkernel}
p_t(x)=\frac{1}{8\pi^{3/2}}\,\delta^{1/2}(x)\,\frac{r(x)}{\sinh r(x)}\,t^{-3/2}\,\nep^{-\frac{r^2(x)}{4t}}   \qquad \forall x\in G\,.
\end{equation}
Note that $p_t=\delta^{1/2}\,h_t$, where $h_t$ is the heat kernel associated with the operator $\mathcal L-I$ and $\mathcal L$ is the Laplace--Beltrami operator on the three-dimensional hyperbolic space. The kernel $h_t$ and its gradient were studied in \cite{ADY, CGHM}.

We remark that given \eqref{heatkernel} it is easy to write the formula for the Poisson kernel. It is enough to replace the factors $t^{-3/2}\,\nep^{-\frac{r^2(x)}{4t}}$ by $t/(t^2+r^2(x))^2$ and the constant factor by  ${1}/{\pi^2}$.

Since $r(x^{-1})=r(x)$ we have that 
\begin{equation}\label{invpt}
p_t(x)=\delta(x)\,p_t(x^{-1})\qquad \forall x\in G\,.
\end{equation}
Via a change of variables in \eqref{conv1}, this implies 
\begin{equation}\label{conv} 
\begin{aligned}
f\ast p_t(x)&=\int f(z^{-1})\,p_t(zx)\dir(z)\\
&=\int f(y)\,p_t(y^{-1}x)\delta(y)\dir(y)\\
&=\delta(x)\int f(y)\,p_t(x^{-1}y)\dir(y)\,.
\end{aligned}
\end{equation}
We shall need the observation that for any $r>0$ 
\begin{equation}\label{supt}
\sup_{t>0}t^{-3/2}\, \nep^{-\frac{r^2}{4t}}\sim r^{-3}\,. 
\end{equation}

We now give some properties of the heat kernel which will be useful in the rest of the paper.
\begin{lemma}\label{Xipt}
For any point $x\in G$ the derivatives of the heat kernel $p_t$ along the vector fields $X_i$, $i=0,1,2,$ are the following:
\begin{itemize}
 \item[(i)] $X_ip_t(x)=p_t(x)\frac{x_i}{\sinh r}\Big( \frac{1}{r}-\frac{\cosh r}{\sinh r}-\frac{r}{2t}\Big)\,$ for  $i=1,2$;
\item[(ii)] $ X_0p_t(x)=p_t(x)\Big( \frac{a}{r\sinh r}-\frac{\cosh r}{r\sinh r}-\frac{a\cosh r}{\sinh^2 r}
+\frac{1}{\sinh ^2r}+\frac{r}{2t}\frac{\cosh r-a}{\sinh r}
 \Big) \,,$
\end{itemize}
where $r$ denotes $r(x)$. Moreover, 
\begin{itemize}
\item[(iii)] $\sup_{t>0}|X_ip_t(x)|\lesssim r^{-4}$ for $i=0,1,2$ and for all $x\in B_1$;
 \item[(iv)] $\sup_{t>0}|X_ip_t(x)|\lesssim \frac{|x|}{a\,r^2\cosh^2 r}$ for $i=1,2$ and for all $x\in B_1^c$;
\item[(v)]  $\sup_{t>0}|X_0p_t(x)|\lesssim \frac{1}{a\,r^3\cosh r}+\frac{1}{r^2\cosh^2 r}$ for all $x\in B_1^c$\,.
\end{itemize}
\end{lemma}
\begin{proof}
Parts (i) and (ii) follow by direct computation from \cite[Lemma 2.1]{SV}. 

The proof of (iii) relies on the fact that $p_t=\delta^{1/2}\,h_t$ which implies that for $i=0,1,2,$ $|X_ip_t|\lesssim \delta^{1/2}\,(|h_t|+|X_ih_t|))$. From the estimates of $|h_t|$ and its gradient proved in \cite[Theorem 5.9, Corollary 5.49]{ADY}, (iii) follows. 

To prove (iv) and (v), it suffices to combine (i), (ii), \eqref{heatkernel} and  \eqref{supt} with $r=r(x)>1$.  
\end{proof}

For suitable functions $F$, we denote by $\|\nabla F\|$ the Euclidean norm
of the vector $\nabla F = (X_0F,\, X_1F,\,X_2F)$.

By applying the previous lemma, we shall estimate the integral  over the complement of a small ball $B_{r_1}$ of the function
\[
G_{r_1}(x) = \sup_{B(x,r_1/2)}\: \sup_{t>0} \,\|\nabla p_t\|.
\]

\begin{lemma}\label{intnablapt}
For any $r_1\in (0,2C_0)$ 
$$
\int_{B_{r_1}^{c}} G_{r_1}  \dir\leq \frac{C}{r_1}\,,
$$
where $C_0$ is the constant which appears in \eqref{C0}. 
\end{lemma}
\begin{proof}
We first derive some pointwise estimates for $G_{r_1}$, closely similar to 
those of Lemma \ref{Xipt} (iii), (iv) and (v). Let 
$x = (x_1, x_2, a) \in B_{r_1}^{c}$. Any point in $B(x,r_1/2)$ can be
written $xy = (x_1+ay_1,\,x_2+ay_2,\, ab)$, where 
$y = (y_1, y_2, b) \in B_{r_1/2}$ so that $|y| \lesssim 1$ and $b \sim 1$. 
The triangle inequality then implies
$|r(xy) - r(x)| \le r_1/2 \le r(x)/2$ and thus $r(xy) \sim r(x)$ and 
$\cosh r(xy) \sim \cosh r(x)$. This allows us to conclude from Lemma 
\ref{Xipt} (iii) that
\[
G_{r_1}(x) \lesssim r^{-4}, \qquad x \in B_1\setminus B_{r_1} \tag{III},
\]
 with $r=r(x)$.  Next, we observe that $|x+ay| \lesssim |x|+a$ and see that (iv) implies
\[
 \sup_{B(x,r_1/2)}\: \sup_{t>0} \,\|X_i p_t\| \lesssim \frac{1 + a^{-1}|x|}{r^2 \cosh^2r}, \qquad i = 1,2 \tag{IV}.
\]
Moreover,  (v) implies
\[
\sup_{B(x,r_1/2)}\: \sup_{t>0} \,\|X_0 p_t\| \lesssim \frac{1}{a\,r^3\cosh r}+\frac{1}{r^2\cosh^2 r} \tag{V}.
\]

Using these three estimates, we shall integrate $G_{r_1}$ in $B_{r_1}^{c}$.
From (III) we get
\begin{equation}
 \label{localnablapt}
\int_{B_1 \setminus B_{r_1}} G_{r_1} \dir \lesssim \int_{B_1 \setminus B_{r_1}}
 r(x)^{-4}   \dir(x) \lesssim \frac{1}{r_1}\,,
 \end{equation}
since $\dir$ is equivalent to the 3-dimensional Lebesgue measure in $B_1$.

It remains to show that the integral over $B_1^c$ is finite. To do so, we split $B_1^c$ into two regions, $D_1=\{x\in B_1^c: |x|^2>a^2+1\}$ and $D_2=B_1^c\setminus D_1$. If $x\in D_1$, then \eqref{metrica} implies that $\cosh r(x)\sim a^{-1}|x|^2$ and $r(x)\sim \log\big(a^{-1}|x|^2\big)$. For $x\in D_2$, one has $\cosh r(x)\sim a+a^{-1}$ and $r(x)\sim \log\big(a+a^{-1}\big)$. 

For $i=1,2$, we apply (IV) to deal with $X_i p_t$. 
In $D_1$ one has $1 \le a^{-1}|x|$, so that 
$$
\begin{aligned}
&\int_{D_1}\sup_{B(x,r_1/2)} \sup_{t>0}|X_ip_t| \dir(x)
\\&\lesssim \int_0^{+\infty} \frac{\di a}{a}\int_{|x|>\sqrt{a^2+1}}
\frac{a^{-1}|x|}{\big(\log(a^{-1}|x|^2)\big)^{2}\,a^{-2}|x|^4}\,\di x_1\di x_2\,.
\end{aligned}
$$
Making the change of variables $(x_1,x_2)=\sqrt{a}\,(v_1,v_2)$, we get 
\begin{equation}\label{iD1}
\begin{aligned}
& \int_0^{+\infty}  \di a   \int_{|v|>\sqrt{a+a^{-1}}}
  \frac{1}{\big(\log |v|^2\big)^2\,\sqrt{a}\,|v|^3\, }\di v_1\di v_2   \\
&\lesssim \int_0^{+\infty}  \frac{1}{  \sqrt{a^2+1} \, 
   \big(\log(a+a^{-1})\big)^{2} }\di a\\ &<\infty\,.
\end{aligned}
\end{equation}
For the region $D_2$, we get
\begin{equation}\label{iD2}
\begin{aligned}
&\int_{D_2} \sup_{B(x,r_1/2)} \sup_{t>0}|X_ip_t| \dir(x) \\&\lesssim \int_0^{+\infty}\frac{\di a}{a} \int_{|x|<\sqrt{a^2+1}}
\frac{1+a^{-1}|x|}{\big(\log(a+a^{-1})\big)^{2}\,(a+a^{-1})^{2}} \di x_1\di x_2\\
&\lesssim \int_0^{+\infty} \frac{1}{\big(\log(a+a^{-1})\big)^{2}\,(a+a^{-1})^{2}}\,(a^2+1)^{3/2}\frac{\di a}{a^2}\\
& <\infty\,.
\end{aligned}
\end{equation}
To estimate the integrals involving $X_0p_t$, we apply (V). The integral over $B_1^c$ of the 
first term to the right in (V) can be computed by means of (\ref{intformula}) as follows:
\begin{equation}\label{0B1c}
\begin{aligned}
\int_{B_1^c} \frac{1}{a\,r(x_1,x_2,a)^3\,\cosh r(x_1,x_2,a)}\dir(x_1,x_2,a) 
&=\int_1^{\infty} \frac{1}{r^3\,\cosh r} r\,\sinh r {\rm{d}}r\\
&\lesssim 1\,.
\end{aligned}
\end{equation}
The last term in (V) is no larger than the right-hand side of (IV),
so its integral is finite by the above.

The estimates \eqref{localnablapt}, \eqref{iD1}, \eqref{iD2}  and \eqref{0B1c}  give the desired conclusion. 
\end{proof}

\section{Proof of part $(i)$ of Theorem \ref{H1L1bdd}} 

We first show that $\mathcal M$ is uniformly bounded on atoms.
\begin{proposition}\label{uniform}
There exists a positive constant $\kappa$ such that for any atom $A$ 
\begin{equation}\label{unif}
\|\mathcal M A\|_1\leq \kappa\,.
\end{equation}
\end{proposition}
\begin{proof}
We first notice that it is sufficient to prove the inequality \eqref{unif} for atoms supported in \CZ sets centred at the identity. 
Indeed, if $B$ is an atom supported in a \CZ set $R$ centred at a point $c_R$, then the function $A$ defined by 
$A(x)=\delta(c_R)^{-1}B(c_Rx)$ is an atom supported in the set $c_R^{-1}\cdot R$ which is centred at the identity. Moreover, 
it is easy to check that $\|\mathcal M B\|_1=\|\mathcal M A\|_1$. 

\smallskip

Let us thus suppose that $A$ is an atom supported in the \CZ set $R=Q\times [\nep^{-r_0},\nep^{r_0}]=[-L/2,L/2]\times [-L/2,L/2]\times [\nep^{-r_0},\nep^{r_0}]$. We study the cases when $r_0<1$ and $r_0\geq 1$. 

\smallskip

\noindent{\emph{Case $r_0<1$.}}  By \eqref{C0} $R\subset B_{C_0r_0}$. Since $r_0<1$, the set $\tilde R=B_{2C_0r_0}$ 
has measure comparable with $\rho(R)$, and we get
\begin{equation}\label{local}
\|\mathcal M A\|_{L^1(\tilde{R})}\leq \rho(\tilde{R})^{1/2}\,\|\mathcal M\|_{L^2\rightarrow L^2}\,\|A\|_2\lesssim 
\rho(R)^{1/2}\,\|\mathcal M\|_{L^2\rightarrow L^2}\,\rho(R)^{-1/2}\lesssim 1\,.
\end{equation}
It remains to consider $\mathcal M A$ on the complement of $\tilde{R}$. For any point $x$ in $\tilde{R}^c$, we obtain from \eqref{conv} and the cancellation condition of the atom  
\begin{equation}
 \label{ppp}
A\ast p_t(x)=\delta(x)\int_R A(y)\,[p_t(x^{-1}y)-p_t(x^{-1})]\dir(y).
\end{equation}
Let $y$ be any point in the ball $B(e,C_0r_0).$ From the Mean Value Theorem in 
$\RR^3$, we obtain
$$
|p_t(y)-p_t(e)| \lesssim\, r_0 \sup_{B(e,C_0r_0)}\|\nabla p_t\|.
$$
Replacing here $p_t$ by its translate $p_t(x^{-1}\cdot)$ and using the left invariance of $\nabla$, we conclude 
$$
|p_t(x^{-1}y)-p_t(x^{-1})| \lesssim\, r_0 \sup_{B(x^{-1},C_0r_0)}\|\nabla p_t\|.
$$
Since $\|A\|_1 \le 1$, we conclude from this and (\ref{ppp}) that
$$
|A*p_t(x)| \lesssim \, r_0 \delta(x) \sup_{B(x^{-1},C_0r_0)}\|\nabla p_t\|,
$$
and so
$$
|\mathcal M A(x)| \lesssim \, r_0 \delta(x) 
\sup_{B(x^{-1},C_0r_0)}\, \sup_{t>0}\|\nabla p_t\|.
$$
We now integrate over $x \in \tilde{R}^c$, inverting the variable $x$. Then
Lemma \ref{intnablapt} with $r_1 = 2C_0r_0$ leads to
$$
\int_{\tilde{R}^c} |\mathcal M A(x)| \dir (x) \lesssim 1,
$$
which ends the case $r_0 < 1$.

\smallskip

\noindent{\emph{Case $r_0\geq 1$.}} In this case we define $R^{**}=
\{ x:~|x|<\gamma\,L,\,\frac{1}{L}<a<\gamma\,L \}$, where $\gamma>1$ is a suitable constant. Then $\rho(R^{**})$ is comparable to 
$\rho(R)$, and there exists a constant $C$ such that
\begin{equation}\label{local2}
\|\mathcal M A\|_{L^1({R}^{**})}\leq \rho({R}^{**})^{1/2}\,\|\mathcal M\|_{L^2\rightarrow L^2}\,\|A\|_2\leq 
C\,\rho(R)^{1/2}\,\|\mathcal M\|_{L^2\rightarrow L^2}\,\rho(R)^{-1/2}\leq C\,.
\end{equation}
It remains to consider $\mathcal M A$ on the complementary set of ${R}^{**}$. We write $(R^{**})^c=\Omega_1\cup\Omega_2\cup\Omega_3$, where
$$
\begin{aligned}
\Omega_1&=\{ x\in G:a<\frac{1}{L},\,|x|  <\gamma\,L\}\,,\\
\Omega_2&=\{x\in G:a>\gamma{L},\,|x|<a\}\,,\\
\Omega_3&=\{x\in G:|x|>\max(a,\gamma L)\}\,.
\end{aligned}
$$
In \eqref{ppp} we write $y=(y_1,y_2,b)\in R$ and $x=(x_1,x_2,a)$ so that $x^{-1}y=\big(a^{-1}(y_1-x_1), a^{-1}(y_2-x_2),a^{-1}b\big)$.  
We now study $\mathcal M A$ on $\Omega_1,\,\Omega_2,\,\Omega_3$ separately.

For $x\in\Omega_1$ and $y\in R$, \eqref{metrica} implies  
$$
\cosh r(x^{-1}y) \geq \frac{a^{-1}b+a^{-1}b^{-1}|x-y|^2}{2} \geq \frac{a^{-1}b}{2}\geq \frac{a^{-1/2}}{2}\geq \frac{\nep}{2}\,,
$$
the last two inequalities since $b\geq \nep^{-r_0}>L^{-1/2}>a^{1/2}$ and $a^{-1/2}>\nep^{r_0}$, by the definitions of $\Omega_1$ and \CZ sets.  
 Therefore, by \eqref{supt}
$$
\begin{aligned}
\sup_{t>0}p_t(x^{-1}y)  &\lesssim  \delta^{1/2}(x^{-1}y)\,r(x^{-1}y)^{-3}\,\frac{r(x^{-1}y)}{\sinh r(x^{-1}y)}\\
&\lesssim (a^{-1}b)^{-1} \frac{(\log a^{-1})^{-2}}{a^{-1}b^{-1}(b^2+|x-y|^2)}\\
&= \frac{a^{2}}{ (\log a^{-1})^{2}(b^2+|x-y|^2)}\,.
\end{aligned}
$$
This implies that for all $x$ in $\Omega_1$,
$$
\begin{aligned}
\sup_{t>0}|A\ast p_t (x)|&= \delta(x)\sup_{t>0}\Big|\int_R A(y)\,p_t(x^{-1}y)\dir(y) \Big|\\
&\lesssim \delta(x)\, \|A\|_{\infty }\int_R  \sup_{t>0}p_t(x^{-1}y)\dir(y)\\
&\lesssim \delta(x)\,\rho(R)^{-1}\frac{a^2}{(\log a^{-1})^2}\times\\
&\times\int_{\nep^{-r_0}}^{\nep^{r_0}}\int_{|y_1|<L/2}\int_{|y_2|<L/2} 
\frac{1}{b^2+|x-y|^2}\di y_1\di y_2\frac{\di b}{b}\\
&\lesssim a^{-2}\,(L^2r_0)^{-1}\frac{a^2}{(\log a^{-1})^2}\int_{\nep^{-r_0}}^{\nep^{r_0}} \log\Big(\frac{L}{b}\Big)\frac{\di b}{b}\\
%&\lesssim   \frac{1}{L^2r_0(\log a^{-1})^2}    \\
&\lesssim \frac{r_0}{L^2(\log a^{-1})^2}\,.
\end{aligned}
$$
By taking the integral over the region $\Omega_1$, we obtain that
\begin{equation}\label{Omega1}
\begin{aligned}
\int_{\Omega_1}\mathcal M A(x)\dir (x)&\lesssim  \frac{r_0}{L^2}\,\int_0^{1/L}\frac{\di a}{a\,(\log a^{-1})^2}\int_{|x|<\gamma L}\di x_1\di x_2\\
&\lesssim \frac{r_0}{\log L}\sim 1\,.
\end{aligned}
\end{equation}

\smallskip
In order to study the integral of $\mathcal M A$ on the remaining regions, we use the cancellation condition of the atom and write 
$$
\begin{aligned}
A\ast p_t(x)&=\delta(x)\int_RA(y)\,[p_t(x^{-1}y)- p_t(x^{-1}\tilde y) +p_t(x^{-1}\tilde y)-p_t(x^{-1})]\dir(y)\,,
\end{aligned}
$$
where for $y=(y_1,y_2,b)\in R$ we write $\tilde y=(y_1,y_2,1)$ and thus 
$y=\tilde{y} \,{\rm{exp}}(\log b\,X_0)  $. Setting $q_1(x)=\sup_{t>0,\,y\in R}|p_t(x^{-1}y)- p_t(x^{-1}\tilde y)|$ and $q_2(x)=\sup_{t>0,\,y\in R}|p_t(x^{-1}\tilde y)-p_t(x^{-1})|$, we conclude that
$$
\mathcal MA(x)\leq \delta(x)\,\big(q_1(x)+q_2(x)\big)\,\int_R|A(y)|\, \dir(y)\leq \delta(x)\,\big(q_1(x)+q_2(x)\big)\,.
$$ 

We first estimate $q_2$. Observe that $\tilde{y}={\rm{exp}}(y_1X_1+y_2X_2) $, so that
\begin{equation*} 
p_t(x^{-1}\tilde y) -p_t(x^{-1})=\int_0^{1}(y_1X_1+y_2X_2)p_t\big(x^{-1} {\rm{exp}}(sy_1X_1+sy_2X_2)  \big)\di s\,.
\end{equation*}
Here $x^{-1} {\rm{exp}}(sy_1X_1+sy_2X_2)=\big( a^{-1}(sy_1-x_1), a^{-1}(sy_2-x_2), a^{-1} \big)$ and 
$$
2\cosh r\big( x^{-1} {\rm{exp}}(sy_1X_1+sy_2X_2) \big)=
a^{-1}+a+a\,a^{-2} |sy-x|^2\,,
$$ 
because of \eqref{metrica}. We consider $0<s<1$ and $y\in R$, which implies that $|y_1|, |y_2|\leq L/2$. 

For $x\in \Omega_2$ we get $|sy-x|\lesssim a$, $\cosh r \big( x^{-1} {\rm{exp}}(sy_1X_1+sy_2X_2) \big)\sim a$ and also $r \big( x^{-1} {\rm{exp}}(sy_1X_1+sy_2X_2) \big)\sim \log a$, so that Lemma \ref{Xipt} (iv) implies that
$$
\big|  (y_1X_1+y_2X_2)p_t\big( x^{-1} {\rm{exp}}(sy_1X_1+sy_2X_2) \big)\big|\lesssim L\frac{a\,a^{-1}\,a}{a^2(\log a)^2}=\frac{L}{a\,(\log a)^2}\,.
$$
Thus
\begin{equation}\label{Omega2q2}
\begin{aligned}
\int_{\Omega_2}\delta(x)\,q_2(x)\dir(x)&\lesssim L\,\int_{\Omega_2}\frac{1}{a^3(\log a)^2}\dir(x)\\
&\lesssim L\,\int_{\gamma L}^{\infty}\frac{1}{a^3(\log a)^2}\int_{|x|<a}\di x_1\di x_2\frac{\di a}{a}\\
&\lesssim L\,\int_{\gamma L}^{\infty}\frac{\di a}{a^2(\log a)^2}\\
&\lesssim \frac{1}{(\log L)^2}\,.
\end{aligned}
\end{equation}

For $x\in\Omega_3$ we have $|x|>\gamma\,L$, so that $|sy-x|\sim |x|\geq a$, and it follows that $\cosh r   \big( x^{-1} {\rm{exp}}(sy_1X_1+sy_2X_2) \big)\sim a^{-1}|x|^2$. Then Lemma \ref{Xipt} (iv) implies
$$
\big|  (y_1X_1+y_2X_2)p_t\big( x^{-1} {\rm{exp}}(sy_1X_1+sy_2X_2) \big)\big|\lesssim L\frac{a\,a^{-1}\,|x|}{a^{-2}|x|^4\,\big(\log(a^{-1}|x|^2)\big)^2}\,.
$$
Thus
\begin{equation}\label{Omega3q2}
\begin{aligned}
\int_{\Omega_3}\delta(x)\,q_2(x)\dir(x)&\lesssim L\,\int_{\Omega_3}\frac{1}{|x|^3\big(\log (a^{-1}|x|^2\big)^2}\di x_1\di x_2\frac{\di a}{a}\\
&\lesssim L\,\int_{|x|>\gamma L}\frac{\di x_1\di x_2}{|x|^3}\int_{a<|x|}
\frac{\di a}{a\big(\log (a^{-1}|x|^2\big)^2}\\
& \lesssim L\int_{|x|>\gamma L}\frac{\di x_1\di x_2}{|x|^3\log |x|}\\
&\lesssim \frac{1}{\log L}\,.
\end{aligned}
\end{equation}
To estimate $q_1$, we write
$$
p_t(x^{-1}y)-p_t(x^{-1}\tilde y)=\log b\,\int_0^{1}X_0p_t\big(x^{-1}\tilde y \, {\rm{exp}}(s\log b\,X_0)  \big)\di s\,.
$$
Here $x^{-1}\tilde y\,{\rm{exp}}(s\log b\,X_0) =\big(a^{-1}(y_1-x_1), a^{-1}(y_2-x_2), 
a^{-1}b^s   \big) $, and so 
$$
2\cosh r\big( x^{-1}\tilde y\,{\rm{exp}}(s\log b\,X_0)  \big)=
a^{-1}b^s+a\,b^{-s}+a\,b^{-s}a^{-2}|y-x|^2\,.
$$

For $x\in\Omega_2$, we get $|y-x|\lesssim a$ and $\cosh r \big( x^{-1}\tilde y\,{\rm{exp}}(s\log b\,X_0)  \big)  \sim a\,b^{-s}$ which implies $r\big( x^{-1}\tilde y\,{\rm{exp}}(s\log b\,X_0)  \big)\sim \log a$ because 
$a>\gamma L>\gamma\nep^{2r_0}$ and $\nep^{-r_0}\leq b^{-s}\leq \nep^{r_0}$\,. Then Lemma \ref{Xipt} (v) implies that
$$
\begin{aligned}
\big| X_0p_t\big(x^{-1}\tilde y\,{\rm{exp}}(s\log b\,X_0)\big)  \big| &\lesssim \frac{1}{a^{-1}b^s\,a\,b^{-s}\,(\log a)^3}+
\frac{1}{a^2\,b^{-2s}(\log a)^2}\\
&\lesssim \frac{1}{(\log a)^3}\,.
\end{aligned}
$$
Thus
\begin{equation}\label{Omega2q1}
\begin{aligned}
\int_{\Omega_2}\delta(x)\,q_1(x)\dir(x)&\lesssim r_0\int_{\Omega_2}\frac{1}{a^2(\log a)^3}\dir(x)\\
&\lesssim  r_0\int_{a>\gamma L}\frac{1}{a^3(\log a)^3} \int_{|x|<a}\di x_1\di x_2\di a\\
&\lesssim r_0\,\frac{1}{(\log L)^2}\\
&\lesssim \frac{1}{r_0}\,.
\end{aligned}
\end{equation}
For $x\in \Omega_3$ we have $|y-x|^2\sim |x|^2$ and $\cosh r\big( x^{-1}\tilde y\,{\rm{exp}}(s\log b\,X_0)  \big)\sim a\,b^{-s}\,a^{-2}|x|^2 $. It follows that $\log(a^{-1}|x|^2)\lesssim r\big( x^{-1}\tilde y\,{\rm{exp}}(s\log b\,X_0)  \big)$. Lemma \ref{Xipt} (v) now implies
$$
\begin{aligned}
\big| X_0p_t\big(x^{-1}\tilde y\,{\rm{exp}}(s\log b\,X_0)\big)\big|&\lesssim \frac{1}{a^{-1}\,b^s\,a^{-1}\,b^{-s}|x|^2(\log(a^{-1}|x|^2)^3} 
+\frac{1}{a^{-2}\,b^{-2s}|x|^4(\log a^{-1}|x|^2  )^2}\\
&\lesssim  \frac{a^2}{|x|^2(\log(a^{-1}|x|^2)^3}+\frac{a^2}{|x|^3(\log(a^{-1}|x|^2)^2}\,;
\end{aligned}
$$
in the last term we have used the inequalities $b^{2s}\leq \nep^{2r_0}<L<|x|$. 
Thus
\begin{equation}\label{Omega3q1}
\begin{aligned}
\int_{\Omega_3}\delta(x)\,q_1(x)\dir(x)&\leq r_0\int_{|x|>\gamma L} \frac{1}{|x|^2}\int_{a<|x|} \frac{\di a}{a(\log(a^{-1}|x|^2)^3} \di x_1\di x_2\\
&+r_0\int_{|x|>\gamma\,L} \frac{1}{|x|^3}\int_{a<|x|} \frac{\di a}{a(\log(a^{-1}|x|^2)^2} \di x_1\di x_2\\
&\leq r_0\int_{|x|>\gamma L} \frac{\di x_1\di x_2}{|x|^2 (\log |x|)^2} \di x_1\di x_2\\
&+r_0\int_{|x|>\gamma\,L} \frac{\di x_1\di x_2}{|x|^3\,\log|x|} \\
&\leq \frac{r_0}{\log L}+\frac{r_0}{L\log L}\\
&\lesssim 1\,.
\end{aligned}
\end{equation}
Combining the above, we obtain the estimate $\int \mathcal M A\dir\lesssim 1$ in the case $r_0\geq 1$, which completes the proof of Proposition \ref{uniform}. 
\end{proof}

We can now prove part (i) of Theorem \ref{H1L1bdd}. 
 
\begin{proof}
Take a function $f\in H^1$, and let us prove that $\mathcal Mf\in L^1$. There exists a sequence of atoms $\{A_j\}_j$ and complex numbers $\{\lambda_j\}_j$ such that 
$f=\sum_{j=1}^{\infty}\lambda_jA_j$ and $\sum_{j=1}^{\infty}|\lambda_j|\leq 2\|f\|_{H^1}$\,. By \eqref{conv}
$$
f\ast p_t(x)
= \int f(y)\,\delta(x)\,p_t(x^{-1}y) \dir(y)\,.
$$
In view of \cite[Theorem 2.1]{CGGM}, the heat kernel $p_t$ is for each $t>0$ a bounded function on  $G$.  Thus for each $t>0$ and $x\in G$ the function $y\mapsto \delta(x)\,p_t(x^{-1}y)$ is in $L^{\infty}$. But then 
$$
\begin{aligned}
\int f(y) \delta(x)\,p_t(x^{-1}y) \dir(y)&= 
\lim_{N\rightarrow \infty}\sum_{j=1}^N\int  \lambda_j\,A_j(y)\,   \delta(x)\,p_t(x^{-1}y) \dir(y)\,,
\end{aligned}
$$
and so
$$
|f\ast p_t(x)|\leq \sum_{j=1}^{\infty}|\lambda_j|\,|A_j\ast p_t(x)|\,.
$$
Taking the supremum in $t$, we get
$$
\mathcal M f(x)\leq  \sum_{j=1}^{\infty}|\lambda_j| \mathcal M A_j(x)\,,
$$
and the result now follows from Proposition \ref{uniform}.  
\end{proof}

\section{Proof of Part $(ii)$ of Theorem \ref{H1L1bdd} }\label{maxchar}

Even though the two statements in Part (ii) are equivalent, we choose to prove both, by means of simple, explicit examples. First we recall the definition and some properties of the space $BMO$ (see \cite{V} for details). 
\begin{definition}
The space $\mathcal{B}\mathcal{M}\mathcal{O}$ is the space of all functions in $L^1_{\rm{loc}}$ such that
$$\sup_{R\in\mathcal R}\frac{1}{\rho(R)}\int_R|g-g_R|\dir <\infty\,,$$
where $g_R$ denotes the mean value of $g$ in the set $R$, and $\mathcal R$ is the family of \CZ sets. The space $BMO$ is $\mathcal{B}\mathcal{M}\mathcal{O}$ modulo the subspace of constant functions. It is a Banach space endowed with the norm 
$$\|g\|_{BMO}=\sup\Big\{\frac{1}{\rho(R)}\int_R|g-g_R|\dir :~R \in\mathcal R\Big\}\,.$$
\end{definition}
The space $BMO$ can be identified with the dual of the Hardy space $H^1$. More precisely, for any $g$ in $BMO$ the functional $\ell$ defined on any atom $A$ by
$$\ell(A)=\int g\,A\dir  \,,$$
extends to a bounded functional on $H^{1}$. Furthermore, any functional in the dual of $H^1$ is of this type, and 
$\|\ell\|_{(H^1)^*}\sim\|g\|_{BMO}$. Given  functions $g$ in $BMO$ and $f$ in $H^1$ we shall denote by $\langle g,f\rangle$ the action of $g$ on $f$ in this duality. 

Aiming at Theorem \ref{H1L1bdd} (ii), we shall construct a family of functions $\{f_L\}_{L>2}$ in $H^1$ such that
$$
\lim_{L\rightarrow +\infty}\frac{\|f_L\|_{H^1}}{\|\mathcal Mf_L\|_1}=+\infty\,.
$$ %This implies that there does not exist a positive %constant $C$ such that 
%$$
%\|f\|_{H^1}\leq C\,\|M f\|_1 \qquad \forall f\in H^1\,.
%$$
Fix $L>2$ and consider the rectangles $R_0=[-1,1]\times[-1,1]\times[\frac{1}{\nep},\nep]$ and 
$R_L=(L,0,1)\cdot R_0=[L-1,L+1]\times[-1,1]\times[\frac{1}{\nep},\nep]$. We then define  $f_L=\chi_{R_L}-\chi_{R_0}$. Obviously $f_L$ is a multiple of an atom, so it lies in the Hardy space $H^1$. We shall estimate the $H^1$-norm of the function $f_L$ from below, by applying the duality between $H^1$ and $BMO$.

\begin{lemma}\label{fLH1}
There exists a positive constant $C$ such that
$$
\|f_L\|_{H^1}\geq C\,\log L\qquad \forall L>2\,.
$$
\end{lemma}
\begin{proof}
We consider the function $h$ in $BMO(\mathbb R)$ given by $h(s)=\log|s|$ for all $s$ in $\mathbb R$ and define $g(x_1,x_2,a)=h(x_1)=\log |x_1|$\,.  First we observe that $g$ is in $BMO$. Indeed, let $R=Q_1\times Q_2\times I$ be a Calder\'on--Zygmund set, where $Q_1$, $Q_2$ and $I$ are intervals of the form given by Definition \ref{Czsets}. One finds that $g_R$ coincides with the mean $h_{Q_1}$ of $h$ in $Q_1$, and so 
$$\frac{1}{\rho(R)}\int_R|g-g_R|\dir=\frac{1}{|Q_1|}\int_{Q_1}|h(x_1)-h_{Q_1}|\di x_1\leq \|h\|_{BMO(\mathbb R)}\,.
$$
This implies that $g$ is in $BMO$ and $\|g\|_{BMO}=\|h\|_{BMO(\mathbb R)}$. 

Since $f_L$ is a multiple of an atom, $\langle g,f_L\rangle=\int g\,f_L\dir$, and it is easy to verify that 
$\Big|\int g\,f_L\dir \Big|\geq C\log L\,.$
On the other hand, $\big|\langle g,f_L\rangle\big|\lesssim \|g\|_{BMO}\,\|f_L\|_{H^1}$. The lemma follows.

\end{proof}

\begin{remark}

The estimate of Lemma \ref{fLH1} is sharp, which can be seen as follows. With $J=\big[(\log  L)/2)\big]$, we shall write the function $f_L$ as a linear combination of $2J+1\sim\log L$ atoms. The intuitive idea is to construct a chain of atoms which reach the same "height" $~L/2$ in the $a$-variable as the geodesics connecting  the sets $R_0$ and  $R_L$. More precisely, for  $j=0,1,...,J-1$ set 
$$
P_j=(0,0, \nep^{2j})\cdot R_0=(-\nep^{2j},\nep^{2j})\times (-\nep^{2j},\nep^{2j})\times (\nep^{2j-1},\nep^{2j+1})\,,
$$
and define the functions  
$$
A_j=\frac{\rho(R_0)}{\rho(P_j)}\chi_{P_j}-\frac{\rho(R_0)}{\rho(P_{j+1})}\chi_{P_{j+1}}\,.
$$
Similarly, set 
$$
Q_j=(L,0, 1)\cdot P_j=(L-\nep^{2j},L+\nep^{2j})\times (-\nep^{2j},\nep^{2j})\times (\nep^{2j-1},\nep^{2j+1})\,,
$$
and define the functions 
$$
B_j=\frac{\rho(R_L)}{\rho(Q_j)}\chi_{Q_j}-\frac{\rho(R_L)}{\rho(Q_{j+1})}\chi_{Q_{j+1}}\,.
$$
Notice that $\rho(P_j)=\rho(Q_j)\sim\,\nep^{4j}$ for each $j$, and it is easy to see that there exist \CZ sets $\tilde{P}_j$ and  $\tilde{Q}_j$ such that $\rho(\tilde{P}_j)=\rho(\tilde{Q}_j)\sim e^{4j}$ and
$$
{\rm{supp}}(A_j)=P_j\cup P_{j+1}\subset \tilde{P}_j\qquad {\rm{and}}
\qquad {\rm{supp}}(B_j)=Q_j\cup Q_{j+1}\subset \tilde{Q}_j\,.
$$
Thus, $A_j$ and $B_j$ are multiples of atoms. Moreover, $P_{J}\cup Q_J$ is contained in a \CZ set whose measure is comparable with $\rho(P_J)$, and thus $\frac{\rho(R_0)}{\rho(P_{J})} \big[ \chi_{P_{J}}-\chi_{Q_{J}}\big]$ is also a multiple of an atom. 

We conclude that
$$
f_L=\sum_{j=0}^{J-1} B_j-\sum_{j=0}^{J-1}A_j-\frac{\rho(R_0)}{\rho(P_{J})} \big[ \chi_{P_{J}}-\chi_{Q_{J}}\big]\,,
$$
so that $\|f_L\|_{H^1}\sim \log L\,$.

\end{remark}

\begin{proposition}\label{MfL}
There exists a positive constant $C$ such that for $L>C$
$$
\|\mathcal M f_L\|_{1}\lesssim \,\log\log L\,.
$$
\end{proposition}
\begin{proof}
Denote by $2R_0$ the rectangle $[-2,2]\times[-2,2]\times[\frac{1}{\nep^2},\nep^2]$ and by $2R_L $ the rectangle $(L,0,1)\cdot (2R_0)$. 
We shall estimate the $L^1$-norm of the maximal function $\mathcal Mf_L$  by integrating it over different regions of the space. 

{\bf{Step 1.}}  The operator $\mathcal M$ is a contraction on $L^{\infty}$, so that $\mathcal M f_L\leq 1$, and since $\rho(2R_0)=\rho(2R_L)\sim 1$, clearly
\begin{equation}\label{int2R-}
\int_{2R_0\cup\, 2R_L}\mathcal M f_L \dir\lesssim 1\,.
\end{equation}

{\bf{Step 2.}} Choose a ball $B=B(e,r_B)$ with $r_B=(\log L)^{\alpha}$, where $\alpha >2$ is a constant. 
Then \eqref{metrica} implies that $B\supset 2R_0\cup 2R_L$ if $L$ is large enough. We shall estimate the maximal function on $B\setminus (2R_L\cup 2R_0)$.  From \eqref{conv} we see that for any $x$ 
\begin{equation}\label{formulaMchi}
\mathcal M \chi_{R_0}(x)=\sup_t\int \chi_{R_0}(y)\,p_t(y^{-1}x)\dil (y)\lesssim 
\sup_{y \in R_0} \sup_t p_t(y^{-1}x) \,.
\end{equation}
If $x\in (2R_0)^c$ and $y\in R_0$, then $\delta^{1/2}(y^{-1}x)\sim \delta^{1/2}(x)$ and 
$$
|r(y^{-1}x)-r(x)|=|d(y,x)-d(x,e)|\leq d(y,e)\leq C\,.
$$
Applying \eqref{heatkernel} and \eqref{supt}, we see that then
$$
\sup_{y \in R_0} \sup_t p_t(y^{-1}x)\lesssim \delta^{1/2}(x)\, \frac{1}{r(x)^2\sinh r(x)}\,.
$$
It follows that for any $x\in (2R_0)^c$ 
\begin{equation}\label{Mchi+}
\mathcal M \chi_{R_0}(x)\lesssim   \delta^{1/2}(x)\, \,\frac{1}{r(x)^2\,\sinh r(x)} \,.
\end{equation}

Since $2R_0$ can be seen to contain the unit ball $B(e,1)$, we apply \eqref{intformula} 
and \eqref{Mchi+}, getting
\begin{equation}\label{intB}
\begin{aligned}
\int_{B\setminus 2R_0} \mathcal M\chi_{R_0}\dir &\lesssim \int_1^{r_B} r^{-2}\,\frac{1}{\sinh r}\,r\,\sinh r\di r\\
&= \log r_B \\
& \lesssim \log\log L\,.
\end{aligned}
\end{equation} 

Observe now that $\mathcal M\chi_{R_L}(x)=\mathcal M\chi_{R_0}(\tau_L x)$, where $\tau_Lx=(-L,0,1)\cdot x$, for any $x$ in $G$. Using the facts that $\tau_L R_L=R_0$ and $\tau_LB\subset B(e,2r_B)=2B$ for large $L$, and changing variable $\tau_Lx=v$, we obtain
\begin{equation}\label{intBB}
\begin{aligned}
\int_{B\setminus 2R_L} \mathcal M\chi_{R_L}(x)\dir(x)&= \int_{B\setminus 2R_L} \mathcal M\chi_{R_0}(\tau_Lx)\dir(x)\\
&=\int_{\tau_LB\setminus \tau_L(2R_L)} \mathcal M\chi_{R_0}(v)\, \delta(-L,0,1) \dir(v)\\
&\leq \int_{2B\setminus 2R_0}\mathcal M\chi_{R_0}(v)\,  \dir(v)\\
&\lesssim \log (2r_B) \\
&\lesssim \log\log L\,, 
\end{aligned}
\end{equation}
as before. 

Thus, from \eqref{intB} and \eqref{intBB} we deduce that
\begin{equation}\label{intBBB}
\int_{B\setminus (2R_0\cup\, 2R_L)}\mathcal Mf_L\dir\lesssim \log\log L\,.
\end{equation}

\bigskip

We now split the complement of $B$ into the following three regions:
$$
\begin{aligned}
\Gamma_1&=\Big\{x=(x_1,x_2,a)\in B^c:\,a<a^*,\,|x|<f(a)\Big\}\,,\\
\Gamma_2&=\Big\{x=(x_1,x_2,a)\in B^c:\,a\geq a^* \Big\}\,,\\
\Gamma_3&=\Big\{x=(x_1,x_2,a)\in B^c:\,a<a^*,\,|x|\geq f(a)\Big\}\,,
\end{aligned}
$$
where $a^*=\nep^{   -r_B/8}$ and $f(a)=\nep^{\sqrt{\log a^{-1}}}$\,.

\smallskip

{\bf{Step 3.}} We first estimate the integral of $\mathcal M f_L$ over the region $\Gamma_1$, and here we use the simple fact that $\mathcal M f_L\leq \mathcal M\chi_{R_L}+\mathcal M\chi_{R_0}$. For any point $x$ in $\Gamma_1$ we have 
$$
\frac{a^{-1}(1+|x|^2)}{2}<\frac{a+a^{-1}+a^{-1}|x|^2}{2} =
\cosh r(x)<\nep^{r(x)} \,,
$$
which implies that
$$
\log a^{-1}\lesssim r(x)
\,.
$$ 
From \eqref{Mchi+} we obtain
\begin{equation}\label{intOmega1}
\begin{aligned}
\int_{\Gamma_1}  \mathcal M\chi_{R_0} \dir 
 &\lesssim \int_0^{a^*}\frac{\di a}{a}\int_{|x|<f(a)} a^{-1} \,[\log a^{-1}]^{-2}\,\frac{\di x}{a^{-1}(1+|x|^2)}\\
&\lesssim \int_0^{a^*}\frac{\di a}{a}   \,[\log a^{-1}]^{-2}\,\log f(a)\\
&\lesssim \int_0^{a^*}\frac{\di a}{a}   \,[\log a^{-1}]^{-3/2}\\
&\lesssim 1\,.
\end{aligned}
\end{equation} 
With a translation argument as in \eqref{intBB}, we see that
\begin{equation}\label{intOmega1bis}
\int_{\Gamma_1}  \mathcal M\chi_{R_L} \dir\leq \int_{\tau_L(\Gamma_1)}  \mathcal M\chi_{R_0} \dir \,. 
\end{equation} 
Now, any point $(x_1,x_2,a)$ in $\tau_L(\Gamma_1)$ satisfies $a<a^*$ and 
$|x|<f(a)+L$, which implies $f(a)>f(a^*)=\nep^{\sqrt{r_B/8 }}=\nep^{(\log L)^{\alpha/2}/\sqrt{8}}>L$, if $L$ is large enough. Then $|x|<2f(a)$, and the right-hand side of \eqref{intOmega1bis} can be estimated as in \eqref{intOmega1}, with $f(a)$ replaced by $2f(a)$. We conclude that
\begin{equation}\label{intOmega1tris}
\int_{\Gamma_1}  \mathcal M\chi_{R_L} \dir\lesssim 1\,.
\end{equation}

{\bf{Step 4.}} In order to estimate the integrals over $\Gamma_2$ and $\Gamma_3$, we first write the convolution $f_L\ast p_t(x)$ at a point $x\in B^c$ as follows:
$$
\begin{aligned}
f_L\ast p_t(x)&=\ 
\int_{R_L}\,p_t\big(y^{-1} x\big)\di\lambda(y)-\int_{R_0}\,p_t(y^{-1}x)\di\lambda(y) \\
&=\int_{R_0}\,\big[p_t\big(y^{-1}(-L,0,1)x\big)-p_t(y^{-1}x)\big]\di\lambda(y)\,.
\end{aligned}
$$
Let now $y^{-1}=(y_1,y_2,b)$ be any point in $(R_0)^{-1}$ and $x=(x_1,x_2,a)$ any point in $B^c$. Then $y^{-1}(-L,0,1)x=y^{-1}(-L,0,1)y\cdot y^{-1}x=(-bL,0,1)\,y^{-1}x$, and the Mean Value Theorem implies
$$
\begin{aligned}
p_t\big(y^{-1}(-L,0,1)x\big)-p_t(y^{-1}x) &=-bL\,\partial_1p_t\big((s,0,1)\,y^{-1}x\big)\\
&=-bL\,\partial_1p_t\big(s+y_1+bx_1,y_2+bx_2,ba\big)\,,
\end{aligned}
$$
for some $s\in (-bL,0)$. We now use the fact that $X_1=a\partial_1$ and the explicit expression for the derivative $X_1p_t$ given by Lemma \ref{Xipt} (i), to obtain
$$
\begin{aligned}
&p_t\big(y^{-1}(-L,0,1)x\big)-p_t(y^{-1}x) \\&=-\frac{1}{8\pi^{3/2}}\,bL\,(ab)^{-1}\,t^{-3/2}\,\nep^{-\frac{r^2}{4t}}\,\Big(-\frac{r^2}{2t\sinh r}+\frac{\sinh r-r\cosh r}{\sinh^2r} \Big)\,\frac{s+y_1+bx_1}{\sinh r}\,\frac{1}{ab}\\
&=-\frac{1}{8\pi^{3/2}}\,\frac{L}{a^2b}\,t^{-3/2}\,\nep^{-\frac{r^2}{4t}}\,\Big(-\frac{r^2}{2t\sinh^2 r}+\frac{\sinh r-r\cosh r}{\sinh^3r} \Big)\, {(s+y_1+bx_1)}\,,
\end{aligned}
$$
where $ r=r\big(      (s,0,1)\,y^{-1}\,x      \big)\gtrsim 1$\,. Taking the supremum in $t$ and applying \eqref{supt}, we deduce that
\begin{equation}\label{suptdifferenza}
\begin{aligned}
\sup_t\big| p_t\big(y^{-1}(-L,0,1)x\big)-p_t(y^{-1}x) \big|&\lesssim \frac{L}{a^2b}\,r^{-3}\,\frac{r}{\sinh^2 r}\,|s+y_1+bx_1|\\
 &\lesssim \frac{L}{a^2}\,r^{-2}\,\frac{1}{\sinh^2 r}\,(L+|x|)\,,
\end{aligned}
\end{equation}
since $y^{-1}\in (R_0)^{-1}$,  $x\in B^c$ and $s\in (- bL,0)$.  By the triangle inequality, 
$$
\begin{aligned}
|r\big((s,0,1)\cdot y^{-1}x\big)-r(x)|&=|d\big(x,y\,\cdot(-s,0,1)\big)-d(x,e)|\\
&\leq d\big(y\,\cdot(-s,0,1),e\big)\\
&\leq d\big(y\,\cdot(-s,0,1),y \big)+d\big(y,e\big)\\
&=d\big((-s,0,1),e\big)+d\big(y,e\big)\\
&\lesssim \log L\,,
\end{aligned}
$$
where we also used \eqref{metrica}.  Since $r(x)>(\log L)^{\alpha}$, we get $r\big((s,0,1)\,y^{-1}\,x\big)\sim r(x)$ and
$$
\sinh^{-2}r\big((s,0,1)\cdot y^{-1}x\big)\lesssim \sinh^{-2}\big[r(x)-\log L\big]\lesssim \frac{L^2}{\sinh^2r(x)}\,.
$$
This and \eqref{suptdifferenza} imply that
$$
\sup_t\big|    p_t\big(y^{-1}(-L,0,1)x\big)-p_t(y^{-1}x) \big|
 \lesssim \frac{L^3}{a^2}\,r(x)^{-2}\,\frac{1}{\sinh^2 r(x)}\,(L+|x|)\,,
$$
which allows us to conclude that
\begin{equation}\label{fLastpt}
\sup_t|f_L\ast p_t(x)|	\lesssim \frac{L^3}{a^2}\,r(x)^{-2}\,\frac{1}{\sinh^2 r(x)}\,(L+|x|) \,.
\end{equation}

Let us now estimate $\mathcal M f_L$ in the region $\Gamma_2$. Using \eqref{fLastpt} and the fact that $|x|\lesssim a^{1/2} \cosh^{1/2} r(x)$, we get
$$
\mathcal Mf_L(x)\lesssim \frac{L^4}{a^2}\,\frac{r(x)^{-2}}{\sinh^2 r(x)}+ \frac{L^3}{a^{3/2}}\,\frac{r(x)^{-2}}{\sinh^{3/2} r(x)}\qquad\forall x\in \Gamma_2\,.
$$
By applying \eqref{intformula}, we then obtain
\begin{equation}\label{intOmega2}
\begin{aligned}
&\int_{\Gamma_2}\mathcal Mf_L \dir\\
&\lesssim (a^*)^{-1}\, L^4\,\int_{r_B}^{\infty}\frac{r^{-2}}{\sinh^2 r}\,r\,\sinh r\di r+
(a^*)^{-1/2}\,L^3\,
\int_{r_B}^{\infty}\frac{r^{-2}}{\sinh^{3/2} r}\,r\,\sinh r\di r\\
&\lesssim (a^*)^{-1}\,L^4\,\nep^{-r_B}+(a^*)^{-1/2}L^3\,\nep^{-r_B/2}\\
&\lesssim 1\,.
\end{aligned}
\end{equation} 

We finally estimate $\mathcal M f_L$ in the region $\Gamma_3$. For any point $x$ in $\Gamma_3$ we have $\cosh r(x)\sim a^{-1}|x|^2$, and $r(x)\sim \log (a^{-1}|x|^2)$, and also $|x|>\nep^{\sqrt{\log(1/a^*)}}=
\nep^{\sqrt{(\log L)^{\alpha}/8}}>L$ for large $L$. Thus \eqref{fLastpt} implies 
$$
\mathcal Mf_L(x)\lesssim \, \frac{L^3}{a^2}\, [\log(a^{-1}|x|^2)]^{-2}\,\frac{|x|}{[a^{-1}|x|^2]^2 }\qquad\forall x\in \Gamma_3\,.
$$
We now integrate over the region $\Gamma_3$, making the change of variables $a^{-1/2}x=y$:
\begin{equation}\label{intOmega3}
\begin{aligned}
\int_{\Gamma_3}\mathcal Mf_L \dir&\lesssim L^3\int_0^{a^*}\frac{\di a}{a^3}\int_{|x|>f(a)}[\log(a^{-1}|x|^2)]^{-2}\,\frac{|x|\di x}{[a^{-1}|x|^2]^2}\\
&=L^3\int_0^{a^*}\frac{\di a}{a^{3/2}}\int_{|y|>a^{-1/2}f(a)}[\log(|y|^2)]^{-2}\,\frac{|y|\di y}{|y|^4}\\
&\lesssim L^3\,\int_0^{a^*}\frac{\di a}{a^{3/2}}\frac{1}{[a^{-1/2}f(a)]\,[\log a^{-1}+\log f(a)^2]^2}\\
&\lesssim  L^3\,\int_0^{a^*}\frac{\di a}{a}\frac{1}{f(a)\,[\log a^{-1}]^2}\\
&=L^3\,\int_0^{a^*}\frac{\di a}{a}\frac{1}{\nep^{\sqrt{\log a^{-1}}}\,[\log a^{-1}]^2}\,.
\end{aligned}
\end{equation}
An easy calculation shows that  
$$
\int_0^{a^*}\frac{\di a}{a}\frac{1}{  \nep^{\sqrt{\log a^{-1}}} \,[\log a^{-1}]^2}
\lesssim \nep^{-\sqrt{\log (a^*)^{-1}}} \,\big[\log (a^*)^{-1}\big]^{-3/2}\,,
$$
and $\log (a^*)^{-1}=(\log L)^{\alpha}/8$. Thus by \eqref{intOmega3} we get
\begin{equation}\label{intOmega3bis}
\begin{aligned}
\int_{\Gamma_3}\mathcal Mf_L \dir& \lesssim 1\,,
\end{aligned}
\end{equation}
since $\alpha >2$. 
 
Summing up the estimates  \eqref{int2R-}, \eqref{intBBB}, \eqref{intOmega1}, \eqref{intOmega1tris},  \eqref{intOmega2}, \eqref{intOmega3bis}, one concludes that
$$
\int \mathcal Mf_L\dir\lesssim \log\log L\,,
$$
and Proposition \eqref{MfL} is proved. 
\end{proof}

By Proposition \ref{MfL} and Lemma \ref{fLH1} it follows that 
$$
\lim_{L\rightarrow +\infty} \frac{\|f_L\|_{H^1}}{\|\mathcal M f_L\|_{1}}=+\infty\,.
$$
Thus there is no converse estimate to that of Theorem \ref{H1L1bdd} (i). 

\smallskip

We can now finish the proof of Theorem \ref{H1L1bdd}\,(ii). Define the function $f$ by
$$
f(x_1,x_2,a) = \frac1{x_1 (\log x_1)^{3/2}}
$$
in the set $\{(x_1,x_2,a):\,x_1 > 3,\, \,|x_2|< 1, \,\,\nep^{-1}<a<\nep\}$,   $f=c_0$ in $R_0$ and
$f=0$ elsewhere, with the constant $c_0$ chosen so that $\int f\,d\rho =0$.

First we show that $f$ does not lie in $H^1$. We truncate the $BMO$ function $g$ from the proof of Lemma \ref{fLH1}, and let for $N=1,2,...$
$$
g_N=\min\big(\max(g,-N),N \big)\,.
$$
These $L^{\infty}$ functions satisfy $\|g_N\|_{BMO}\leq \|g\|_{BMO}$\,. It is easy to see that
\begin{equation}\label{divergent}
\lim_{N\rightarrow \infty}\int g_N\,f\dir=+\infty\,. 
\end{equation}
If $f$ were in $H^1$ and thus had an atomic decomposition converging in $L^1$, we would get $\langle g_N,f\rangle=\int g_N\,f\dir$. This would be a contradiction, since the $g_N$ define uniformly bounded functionals on $H^1$.

\smallskip

To estimate the $L^1$ norm of $\mathcal{M}f$,
we approximate $f$ with a linear combination of the functions $f_L$ with even, integer values of $L$. Let
$$
\tilde{f} = \sum_{k=2}^\infty c_k f_{2k}\,,
$$
where $c_k$ is the mean value of $f$ in $R_{2k}$. For every $(x_1,x_2,a)$ in $R_{2k}$
$$
\big|f(x_1,x_2,a)-c_k \big|\lesssim x_1^{-2} (\log x_1)^{-3/2} \lesssim k^{-2}\,.
$$
This means that the
restriction of the difference $f-\tilde{f}$ to $R_{2k}$ is $k^{-2}$
times an atom. Observe also that $f$ and $\tilde{f}$ coincide in $R_0$. 

From Proposition \ref{MfL} and Theorem  \ref{H1L1bdd} (i), we conclude
$$
\|\mathcal{M}f\|_1 \le  \|\mathcal{M}\tilde{f}\|_1 +  
\|\mathcal{M}(f-\tilde{f})\|_1 \lesssim
\sum_{k=2}^\infty  c_k \log \log k + \sum_{k=2}^\infty k^{-2}.
$$
Since $c_k\lesssim k^{-1} (\log k)^{-3/2}$, these quantities are finite. Theorem \ref{H1L1bdd} is completely proved.

\end{document}